\documentclass[12pt]{article}
\usepackage{amsmath,amssymb,amsfonts,amsthm}
\newcounter{item}[section]
\newcounter{kirshr}
\newcounter{kirsha}
\newcounter{kirshb}
\newenvironment{enumroman}{\setcounter{kirshr}{1}
\begin{list}{(\roman{kirshr})}{\usecounter{kirshr}} }{\end{list}}
\newenvironment{enumarab}{\setcounter{kirshb}{1}
\begin{list}{(\arabic{kirshb})}{\usecounter{kirshb}} }{\end{list}}
\newtheorem{theorem}{Theorem}[section]

\newtheorem{corollary}[theorem]{Corollary}
\newenvironment{demo}[1]{\noindent{\bf #1.}\upshape\mdseries}
{\nopagebreak{\hfill\rule{2mm}{2mm}\nopagebreak}\par\normalfont}
\theoremstyle{definition}

\newtheorem{example}[theorem]{Example}
\newtheorem{definition}[theorem]{Definition}
\def\R{\mathbb{R}}

\def\C{{\mathfrak{C}}}

\def\Nr{{\mathfrak{Nr}}}

\def\RCA{{\sf RCA}}

\def\A{{\mathfrak{A}}}
\def\B{{\mathfrak{B}}}
\def\C{{\mathfrak{C}}}

\def\M{{\mathfrak{M}}}
\def\N{{\mathfrak{N}}}

\def\CA{{\sf CA}}

\def\Uf{{\sf Uf}}

\def\RCA{{\sf RCA}}

\def\(R)RA{{\bf (R)RA}}
\def\RA{{\bf RA}}
\def\RRA{{\bf RRA}}

\def\R{\mathbb{R}}

 \def\CA{{\sf CA}}
\def\B{{\sf B}}
\def\G{{\sf G}}

 \def\Cm{{\mathfrak{Cm}}}
\def\Nr{{\mathfrak{Nr}}}

\def\Ra{{\mathfrak{Ra}}}

\def\Ra{{\mathfrak{Ra}}}
\def\Nr{{\mathfrak{Nr}}}
\def\Tm{{\mathfrak{Tm}}}
\def\A{{\mathfrak{A}}}
\def\B{{\mathfrak{B}}}
\def\C{{\mathfrak{C}}}

\def\A{{\mathfrak{A}}}
\def\B{{\mathfrak{B}}}
\def\C{{\mathfrak{C}}}

\def\Bb{{\mathfrak{Bb}}}

\def\Bb{{\mathfrak{Bb}}}

\def\CA{{\sf CA}}
\def\RA{{\sf RA}}
\def\RRA{{\sf RRA}}
\def\RCA{{\sf RCA}}

\def\G{{\bf G}}

\def\At{{\sf At}}

\def\Ra{{\sf Ra}}
\def\R{{\sf R}}

\def\Cof{{\sf Cof}}

\title{Blow up and Blur constructions in algebraic logic}
\author{Tarek Sayed Ahmed}
\begin{document}
\maketitle

\begin{abstract} We give a simpler proof of a result of Hodkinson in the context of a blow and blur up construction
argueing that the idea at heart is similar to that adpted by Andr\'eka et all \cite{sayed}
The idea is to blow up a finite structure, replacing each 'colour or atom' by infinitely many, using blurs
to  represent the resulting term algebra, but the blurs are not enough to blur the structure of the finite structure in the complex algebra. 
Then, the latter cannot be representable due to a {\it finite- infinite} contradiction. 
This structure can be a finite clique in a graph or a finite relation algebra or a finite 
cylindric algebra. This theme gives example of weakly representable atom structures tthat are not strongly
representable. This is the essence too of construction of Monk like-algebras, one constructs graphs with finite colouring (finitely many blurs),
converging to one with infinitely many, so that the original algebra is also blurred at the complex algebra level, 
and the term algebra is completey representable, yielding a representation of its completion the 
complex algebra. 

A reverse of this process exists in the literature, it builds algebras with infinite blurs converging to one with finite blurs. This idea due to 
Hirsch and Hodkinson, uses probabilistic methods of Erdos to construct a  sequence of graphs with infinite  chromatic 
number one that is $2$ colourable. This construction, which works for both relation and cylindric algebras,
further shows that the class of strongly representable atom structures
is not elementary.
\end{abstract}


\section{Introduction}

The idea is to blow up a finite structure, replacing each 'colour or atom' by infinitely many, using blurs
to  represent the resulting term algebra, but the blurs are not enough to blur the structure of the finite structure in the complex algebra. 
Then, the latter cannot be representable due to a {\it finite- infinite} contradiction. 
This structure can be a finite clique in a graph or a finite relation algebra or a finite 
cylindric algebra.

We discuss the possibility of obtaining stronger results concerning completions, 
for example we approach the problem as to whether classes of subneat reducts are 
closed under completions, and analogous results 
for infinite dimensions.  Partial results in this direction are obtained by Sayed Ahmed, some of which will be 
mentioned below.

The main idea is to {\it split and blur}. Split  what? You can split a clique by taking $\omega$ many disjoint copies of it, 
you can split  a finite relation algebra, 
by splitting each atom into $\omega$ many, you can split a finite cylindric algebra. Generally, the splitting has to do with blowing up a finite structure
into infinitely many.

Then blur what? On this split one adds a subset of  a set of fixed in advance blurs, usually finite,
and then define an infinite atom structure, induced by the properties of the finite structure he originally started with.
It is not this atom structure that is blurred but rather the original finite structure.
This means that the term algebra built on this new atom structure, 
that is the algebra generated by the atoms,
coincides with a carefully chosen partition of the set of atoms obtained after splittig and bluring
up to minimal  deviations, so the original finite relation algebra is blurred to the extent that is invisible on this level.

The term algebra will be representable, using all such blurs as colours,
But the original algebra structure re-appears in the completion of this term algebra, that is the complex algebra of the atom structure,
forcing it to be
non representable, due to a finite-infinite discrepancy. However, if the blurs are infinite, then, they will blur also the structure
of the small algebra in the complex algebra, and the latter will be representable, inducing a complete representation of 
the term algebra.

\section{Main definition and examples}

We start by giving rigouous definitions of blowing up and bluring a finite structure.
In what follows, by an atom structure, we mean an atom structure of any class of completely additive Boolean algebras.

Let $N$ be a graph, in our subsequent investigations $N$ will be finite. But there is no reason to impose restriction on our next definition, which we try
keep as general as possible. By induce, we mean define in a natural way, and we keep natural at this level of ambiguity.

\begin{definition}
\begin{enumarab}
 
\item A splitting of $N$ is a disjoint union  $N\times I$, where $I$ is an infinite set. 

\item  A blur for $N$ is any  set  $J$. 

\item An atom structure $\alpha$ is blown up and blurred if, there exists a subset $J'$ of a set $J$ of blurs, possibly empty,
such that  $\alpha$ has underlying set $X=N\times I \times J'$; the latter atom structure is called a blur of $N$ via $J$, 
and is denoted by $\alpha(N,J).$  Furthermore, every $j\in J$, induces a non-principal ultrafilter in $\wp(X)$.

\item An atom structure $\alpha(N, J)$ reflects $N$, if $N$ is faithfully represented in $\Cm\alpha(N,J)$

\item An atom structure $\alpha(N,J)$ is  weak if $\Tm\alpha(N,J)$ is representable.

\item An atom structure $\alpha(N, J)$ is very weak $\Tm\alpha(N,J)$ is not representable.

\item An atom structure $\alpha(N,J)$ is strong if  $\Cm\alpha(N,J)$ is representable.
\end{enumarab} 

\end{definition}


We give two examples of weak atom structures.
The first construction builds two relativized 
set algebras based on a certain model that is in turn a Fraisse limit of a class of certain
labelled graphs, with the labels coming from $\G\cup \{\rho\}\times n$, where $\G$ is an arbitrary graph and $\rho$ is a new colour. 
Under certain conditions on $\G$, the first set algebra can be represented on square units, the second, its completion, cannot.

\section{First example}


\subsection{The cylindric algebra}

Let $\G$ be a graph. One can  define a family of labelled graphs $\cal F$ such that every edge of each graph $\Gamma\in {\cal F}$,
is labelled  by a unique label from
$\G\cup \{\rho\}\times n$, $\rho\notin \G$, in a carefully chosen way. The colour of $(\rho, i)$ is
defined to be $i$. The \textit{colour} of $(a, i)$ for $a \in \G$  is $i$.
$\cal F$ consists of all complete labelled graphs $\Gamma$ (possibly
the empty graph) such that for all distinct $ x, y, z \in \Gamma$,
writing $ (a, i) = \Gamma (y, x)$, $ (b, j) = \Gamma (y, z)$, $ (c,l) = \Gamma (x, z)$, we have:\\
\begin{enumarab}
\item $| \{ i, j, l \} > 1 $, or
\item $ a, b, c \in \G$ and $ \{ a, b, c \} $ has at least one edge
of $\G$, or
\item exactly one of $a, b, c$ -- say, $a$ -- is $\rho$, and $bc$ is
an edge of $\G$, or
\item two or more of $a, b, c$ are $\rho$.
\end{enumarab}

One forms a labelled graph $M$ which can be viewed as model of a natural signature,
namely, the one with relation symbols $R_{(a, i)}$, for each $a \in \G \cup \{\rho\}$, $i<n$ and

Then one takes a subset $W\subseteq {}^nM$, by roughly dropping assignments that do not satify $(\rho, l)$ for every $l<n$.
Formally, $W = \{ \bar{a} \in {}^n M : M \models ( \bigwedge_{i < j < n,
l < n} \neg (\rho, l)(x_i, x_j))(\bar{a}) \}.$
Basically, we are throwing away assignments $\bar{a}$ whose edges betwen two of its elements are labelled by $\rho$, and keeping those
whose edges of its elements are not.
All this can be done with an arbirary graph.

Now for particular choices of $\G$; for example if $\G$ is a certain rainbow graph, or more simply a countable infinite collection of pairwise
union of disjoint $N$ cliques with $N\geq n(n-1)/2$, or  is the graph whose nodes are the natural numbers, and the edge relation is defined by
$iEj$ iff $0<|i-j|<N$, for same $N.$ Here, the choice of $N$ is not haphazard, but it a bound of edges of complete graphs having $n$ nodes.

The relativized set algebras based on $M$, but permitting as assignments satisfying formulas only $n$ sequences in $W$ will be an atomic
representable algebra.

This algebr, call it $\A$, has universe $\{\phi^M: \phi\in L^n\}$ where $\phi^M=\{s\in W: M\models \phi[s]\}.$ (This is not representable by its definition
because its unit is not a square.) Here $\phi^M$ denotes the permitted asignments satisfyng $\phi$ in $M$.
Its completion is the relativized set algebra $\C$ with universe the larger $\{\phi^M: \phi\in L^n_{\infty,\omega}\}$,
which turns out not representable. (All logics are taken in the above signature).
The isomorphism from  $\Cm\At\A$ to $\C$ is given by $X\mapsto \bigcup X$.

Let us formulate this construction in the context of split and blur.
Take the $n$ disjoint copies of $N\times \omega=\G$.
Let $a\in \G\times n$. Then $a\in N\times \omega\times n$. 
Then for every $(a,i)$ where $a\in N\times \omega$, and $i<n$, we have an atom $R_{a,i}^{\M}\in \A$. 
The term algebra of $\A$ is generated by those. 

Hence $\N\times \omega\times n$ is the atom structure of $\A$ which can be weakly represented using the $n$ blurs, namely the set
$\{\rho, i): i<n\}$. 
The clique  $N$ appeas on the complex algebra level, forcing a finite $N$ colouring, 
so that the complex algebra cannot be representable.

We note that if $N$ is infinite, then the complex algebra (which is the completion of the algebra constructed
as above ) will be representable and so $\A$, together the term algebra will be  completely 
representable.

\subsection{The relation algebra}

We use the graph $N\times \omega$ of countably many disjoint $N$ cliques.
We define a relation algebra atom structure $\alpha(\G)$ of the form
$(\{1'\}\cup (\G\times n), R_{1'}, \breve{R}, R_;)$.
The only identity atom is $1'$. All atoms are self converse, 
so $\breve{R}=\{(a, a): a \text { an atom }\}.$
The colour of an atom $(a,i)\in \G\times n$ is $i$. The identity $1'$ has no colour. A triple $(a,b,c)$ 
of atoms in $\alpha(\G)$ is consistent if
$R;(a,b,c)$ holds. Then the consistent triples are $(a,b,c)$ where

\begin{itemize}

\item one of $a,b,c$ is $1'$ and the other two are equal, or

\item none of $a,b,c$ is $1'$ and they do not all have the same colour, or

\item $a=(a', i), b=(b', i)$ and $c=(c', i)$ for some $i<n$ and 
$a',b',c'\in \G$, and there exists at least one graph edge
of $G$ in $\{a', b', c'\}$.

\end{itemize}
$\alpha(\G)$ can be checked to be a relation atom structure. It is exactly the same as that used by Hirsch and Hodkinson, except
that we use $n$ colours, instead of just $3$. This allows the relation algebra to have an $n$ dimensional cylindric basis
and, in fact, the atom structure of $\A$ is isomorphic (as a cylindric algebra
atom structure) to the atom structure ${\cal M}_n$ of all $n$-dimensional basic
matrices over the relation algebra atom structure $\alpha(\G)$.

Indeed, for each  $m  \in {\cal M}_n, \,\ \textrm{let} \,\ \alpha_m
= \bigwedge_{i,j<n}  \alpha_{ij}. $ Here $ \alpha_{ij}$ is $x_i =
x_j$ if $ m_{ij} = 1$' and $R(x_i, x_j)$ otherwise, where $R =
m_{ij} \in L$. Then the map $(m \mapsto
\alpha^W_m)_{m \in {\cal M}_n}$ is a well - defined isomorphism of
$n$-dimensional cylindric algebra atom structures.

It can be  shown that the complex algebras of this atom structure is not representable, because its chromatic number is finite; indeed
it is exactly $N$. (This will be demonstrated below.)

But we want more. Is it possible, thatthe constructed relation algebrais {\it not} in $S\Ra \CA_{n+2}$ which is strictly smaller that $\RRA$.
The idea that could work here, is to use {\it relativized} representations. 
Algebras in $S\Ra\CA_{n+2}$ do posses representations
that are only {\it locally} square.  So is the blurring, using $n$ colours, based on $N$, namely $(\rho, i)$ $i<n$,
enough to prohibit the complex algebra to be representable in a {\it weaker} sense, which means that we have to strengthen our conditions, involving
the superscrit $2$ in the equation with $N$ and $n$. We have $N\geq n(n-1)/2$ but we need a further combinatorial property relating the triple
$(2, N, n)$

In any event, there is a finite-infinite discrepancy here, as well, no matter what kind of representation we consider,
the base has to be infinite. A representation maps the complex algebra into the powerset of a set of ordered pairs, with base $X$,
thae latter has to be infinite.  At the same time the graph has an $N$ coloring, 
and this can be used to partition the complex algebra into $(N\times n)+1$ blocks. 

But this is not enough; the idea in the classical case,
works because  one member of the partition induced by the finite colouring will be monochromatic, 
and will satisfy $(P;P)\cdot P\neq 0$, which is a contradiction.

The last condition is not guaranteed  when we have only relativized representations, because if $h$ 
is such a representation, it is not really a faithful one, in the sense that it can happen that there are $x_0, x_1,x_2\in X$,
and $(x_0,x_1)\in h(a)$, $(x_1, x_2)\in h(b)$, $(x_0, x_2)\in h(c)$, and $a, b, c\in \Cm\G$,  but $h((a;b).c)=0$
if the node $x_1$ witnessing composition, lies outside the $n$ clique determined by $x_0, x_2$,
This cannot happen in case of  classical representation. 
Finite clique is the measure of {\it squareness}. It will be defined shortly.

But we are also {\it certain } that the complex algebra is not in $S\Ra\CA_{n+k}$ for some $k\in \omega$, by the neat embedding 
theorem for relation algebras, namey,  
$\RRA=\bigcap_{k\in \omega}S\Ra\CA_{n+k}.$ 

Now, accordingly, let us keep $k$ loose, for the time being. We want to determine the least such $k$.
Remember that we required that $N\geq n(n-1)/2$, this was necessary to show that permutations of $\omega\cap \{\rho\}$ 
induces $n$ back and forth systems of partial isomorphisms of size less than $n$ in  our limiting labelled graph $M$, showing that is 
{\it strongly}
$n$ homogeneous, when viewed as a model for the language $L$. This in turn enabled us to show
that the term algebra is representable. 

The plan is to go on with the proof and see what other combinatorial properties one should impose on the relationship between
$N$, $n$ and now $k$ to prohibit even {\it a relativized representation}. 
Obviously one should  keep the condition $N\geq n(n-1)/2$ not to tamper with the
first part of the proof.

Let $\A=\Cm\G$, and assume that $V\subseteq X\times X$ is a relativized representation. 
An arbitray relativized representation, 
that is if we take any set of ordered pairs, is useless, its not what we want.

We need {\it locally} square representations that are like representations only on finite cliques of the base.
But what does locally square mean? 
A clique $C$ of $X$ is a subset of the domain $X$, that can indeed be viewed as a complete graph, in the sense that any two points in it
we have $X\models 1(x,y)$, equivalently $(x,y)\in V$, where $V$ is the unit of the relativization.
The property of $n+k$ squareness means, then for all cliques $C$ of $X$ with $|C|<n+k$, can always be extended to another clique 
having at most one more element witnessing composition, so that composition can be preserved in the representation,
but only locally. It is easier to build such representations; from the game theoretic point of view because $\forall$ 
moves are restricted by the size of cliques, which means
that the chance that exists provide a node witnessing composition is higher.



Now lets getting starting with our plan.

Assume for contradiction that $\Cm\alpha(\G)\in S\Ra \CA_{n+k}$, and $k\geq 2$. 
Then $\Cm\alpha(\G)$ has an $n+k-2$-flat representation $X$ \cite{HHbook2}  13.46, 
which is $n+k-2$ square \cite{HHbook2} 13.10. 

In particular, there is a set $X$, $V\subseteq X\times X$ and $g: \Cm\alpha(\G)\to \wp(V)$ 
such that $h(a)$ ($a\in \Cm\alpha(\G)$) is a binary relation on $X$, and
$h$ respects the relation algebra operations. Here $V=\{(x,y)\in X\times X: (x,y)\in h(1)\}$, where $1$ is the greatest element of 
$\Cm\alpha(\G)$. We write $1(x,y)$ for $(x,y)\in h(1).$ 

For any $m<\omega$, let $C_m(X)=\{\bar{a}\in {}^mX: Range(a)\text { is an $m$ clique }\}$, then $n+k-2$ squareness
means that that if $\bar{a}\in C_{n+k-2}(X)$, $r,s\in \Cm\G$, $i,j,k<n, k\neq i,j$, and $X\models (r;s)(a_i, a_j)$ then there is $b\in C_{n+k-2}(X)$ 
with $\bar{b}$ agreeing with $\bar{a}$ except possibly at $k$
such that $X\models r(b_i, b_k)$ and $X\models s(b_k, b_j)$. 

This is the definition. But it is not hard to show that this is equivalent to
 the simpler condition that 
for all cliques $C$ of $X$ with $|C|<n+k$, all $x,y\in C$ 
and $a,b\in \Cm\alpha(\G)$, $X\models (a;b)(x,y)$ 
there exists $z\in X$ such that $C\cup \{z\}$ is a clique and $X\models a(x,z)\land b(z,y)$.

Now $\G$ has a finite colouring using $N$ colours. Indeed, the map $f:N\times \omega$ defined by $f(l,i)=l$ 
is a finite colouring using $N$ colours. For $Y\subseteq N\times \omega$ and $l<n$ define
$(Y, k)=\{(a,i,l): (a,i)\in Y\}$, regarded as a subset of $\Cm\G$. 

The nodes of $N\times \omega$ can be partitioned into
sets $\{C_j: j<n\}$ such that there are no edges within $C_j$.
Let $J=\{1', (C_j,k): j<N, k<n\}$ 
Then clearly, $\sum J=1$ in $\Cm\alpha(\G)$,  so that $J$ is partition of $\Cm\alpha(\G)$ into $N\times n+1$ blocks.

As $J$ is finite, we have for any $x,y\in X$ there is a $P\in J$ with
$(x,y)\in h(P)$. Since $\Cm\alpha(\G)$ is infinite then $X$ is infinite. 
Ramsys's theorem aplies in this context, to allow us to infer, that there are distinct
$x_i\in X$ $(i<\omega)$, $J\subseteq \omega\times \omega$ infinite
and $P\in J$ such that $(x_i, x_j)\in h(P)$ for $(i, j)\in J$, $i\neq j$. 
Then $P\neq 1'$. 

The condition we need on $k$, is that if $(x_0, x_1)\in h(a)$, $(x_1, x_2)\in h(b)$ and 
$(x_0, x_2)\in h(c),$ then $a;b.c\neq 0$.

So this prompts: 

{\it Find a combinatorial relation between $n,k, N$ with $N\geq n(n-1)/2$
that forces $(P;P)\cdot P\neq 0$. What is the least such $k$? This is formulated for any $P$, but maybe the condition
would also force Ramseys theorem to give the right 
block.} 


A non -zero element $a$ of $\Cm\alpha(\G)$ is monochromatic, if $a\leq 1'$,
or $a\leq (\Gamma,s)$ for some $s<n$. 
Now  $P$ is monochromatic, and the $C_j$ s are independent, it follows also from the definition of $\alpha$ that
$(P;P)\cdot P=0$. 

$\Cm\alpha(\G)$ is not in $S\Ra\CA_{n+m}$, and from this, it will follow that  $\Cm\M_n\notin S\Nr_n\CA_{n+m}$, for al $m\geq k$.
Showing that the latter two cases are not closed under completions.


For a relation algebra $\R$ having an $n$ dimensional cylindric basis, let ${\sf Mat}_n\R$ be the term cylindric algebra 
of dimension $n$ generated by the basic matrices.

\begin{theorem} Let $\G$ be a graph that is a disjoint union of cliques having size $n$. Then there is a strongly $n$ homogenious
labelled graph $M$, every edge is labelled by an element from $\G\cup \{\rho\}\times n$,
$W\subseteq {}^nM$, such that the set algebra based on $W$ is an atomic $\A\in \RCA_n$, and there is an atomic 
$\R\in \RRA$ the latter with an $n$ dimensional cylindric basis, 
such that $\A\cong {\sf Mat}_n\R$, and the completions of $\A$ and $\R$ are not representable, hence they are not completely
representable.
\end{theorem}

\section{Example}

\subsection{The relation algebra} 

Here we turn to our second split and blur construction.
It is a simplified version of the proof of Andr\'eka and N\'emeti, except that for a set of blurs $J$, 
they defined infinitely many tenary relations on $\omega$ with suffixes from $J$, to synchronize the composition operation. 
This was necessary to show
that the required algebras are generated by a single element; here we  use one uniform relation, 
and we sacrifize with this part of the result, which is worthwhile, due to
the reduction of the complexity of the proof. We think that
our simplified version captures the essence
of  the blow  up and blur construction of Andr\'eka and N\'emeti.

Let $I$ and $J$ be sets, for the time being assume they are finite.
We will define two partitions $(H^P: P\in I)$ and $(E^W: W\in J)$ of a given infinite set $H$,
using atoms from a finite relation algebra for the first superscripts, and "blurs" (literally) for the second superscript.

The blurs do two things. They are just enough to distort the structure of $\bold M$ in the term algebra, but not in its completions,
but at the same time they are colours that are necessary 
for representing the term algebra.

Indeed, we use the first partition to show that the complex algebra of our atom structure is not representable,
while we use the second to show that the term
algebra is representable.


Let us start getting more concrete. Let $I$ be a finite set with $|I|\geq 6$. Let $J$ be the set of all $2$ element
subsets of $I$, and let
$$H=\{a_i^{P,W}: i\in \omega, P\in I, W\in J, P\in W\}.$$
In a minute we will get even more concrete by choosing a specific finite relation relation $\bold M$ with certain properties, namely,
it cannot be represented on infinite sets. The atoms of $\bold M$ will be  $I$. This algebra is finite, so it cannot do what we want. 
A completion of a finite algebra is itself.

The index $i$ here  says that we will replace each atom of this relation algebra by infinitely many atoms, that will define an atom structure
of a new infinite relation algebra, {\it the desired algebra}.
(This is an instance of a technique called {\it splitting}, which involves splitting an atom into smaller atoms.
Invented by Andreka, it is very useful in proving non representability  results).

The structure of $\bold M$ will be {\it blown up} by splitting the atoms, then 'blurred' in the term algebra, but it will {\it not} be blurred
in the  completion of the term algebra. More precisely, $\bold M$  will be a subalgebra of the completion,
but it may (and will not be) a subalgebra
of the term algebra.

The best way to visualize the partitions we will define is to imagine that the atoms of the new algebras, form a partition of
an infinite rectangle with finite base $I$ and side $\omega$ reflecting the infinite splitting of $I$.
Or to view it as an infinite tenary martrix,  with each entry indexed by $(i, P,W)\in \omega\times I\times J$, $P\in W$.

We now define two finite partitions of the rectangle, namely $H$.
For $P\in I$, let
$$H^P=\{a_i^{P,W}: i\in \omega, W\in J, P\in W\}.$$
The finite relation algebra will be embedable in the completion via $P\mapsto H^P$,
no distortion involved. $\bold M$ will still be up there on the global level.

The $J$s are the blurs, for $W\in J,$ let
$$E^W=\{a_i^{P,W}: i\in \omega, P\in W\}.$$
The singletons will generate this partition up to a `finite blurring'.
That is the term algebra will consist of all those $X$ such that $X$ intersects $E^W$ finitely or cofinitely.
For each $W\in J$, we have $W\subseteq I$, and so $E^W$ will be the subrectangle of $H$ on the base $W$.

To implement our plan we further need  a tenary relation, which synchronizes composition; it will
tell us  which rows in the rectangle, allow composition like $\bold M$.

For $i,j,k\in \omega$ $e(i,j,k)$ abbreviates that $i,j,k$ are {\it evenly distributed}, i.e.
$$e(i,j,k)\text { iff } (\exists p,q,r)\{p,q,r\}=\{i,j,k\}, r-q=q-p$$
For example $3,5,7$ are evenly distributed, but $3,5,8$ are not.
All atoms are self-converse. This always makes life easier.
We define the consistent triples as follows
(Involving identity are as usual $(a, b, Id): a\neq b).$

Let $i,j,k\in \omega$, $P,Q,R\in I$ and $S,Z,W\in J$ such that
$P\in S$, $Q\in Z$ and $R\in W$. Then the triple
$(a_i^{P,S},a_j^{Q,Z}, a_k^{R,W})$ is consistent iff
either
\begin{enumroman}
\item $S\cap Z\cap W=\emptyset,$
or
\item $e(i,j,k)\&P\leq Q;R.$
\end{enumroman}
The second says that if $i,j,k$ are $e$ related then the composition of $P$, $Q$ and $R$
existing on those three rows, is defined like $\bold M$.

Let $\cal F$ denote this atom structure, ${\cal F}=H\cup \{Id\}$

Now, as promised, we choose a (finite) relation algebra $\bold M$ with atoms $I\cup \{1d\}$ such that
for all $P,Q\in I$, $P\neq Q$ we have
$$P;P=\{Q\in I: Q\neq P\}\cup \{Id\}\text { and } P;Q=H$$
Such an $\bold M$ exists It is also known that
$\bold M$, if representable, can be only represented on finite sets. 
Now using the above partitions we show:
\begin{theorem}
\begin{enumarab}
\item $\Cm{\cal F}$ is a relation algebra that is not representable.
\item ${\cal R}$ the term algebra over $\cal F$ is representable.
\end{enumarab}
\end{theorem}
\begin{proof}
\begin{enumarab}
\item Non representabiliy uses the first partition of $H$. Note that $;$ is defined on $\Cm({\cal F})$ so  that
$$H^P; H^Q=\bigcup \{H^Z: Z\leq P; Q\in {\bold M}\}.$$
So $\bold M$ is isomorphic to a subalgebra of $\Cm F$. But $\Cm F$
can only be represented on infinite sets, while $\M$ only on finite ones, hence we are done.

\item The representability of the term algebra uses the second partition.
The blow up and blur algebra is ${\cal R}=\{X\subseteq F: X\cap E^W\in Cof(E^W), \forall W\in J\}$.
For any $a\in F$ and $W\in J$, let
$$U^a=\{X\in R: a\in X\}$$
and
$$U^W=\{X\in R: |Z\cap E^W|\geq \omega\}$$
Let
$$\Uf=\{U^a: a\in F\}\cup \{U^W: W\in J: |E^W|\geq \omega \}.$$
$\Uf$ denotes the set of ultrafilters of $\cal R$, that include at least one non-principal
ultrafilter, that is an element of the form $U^W$.

Let $F,G,K$ be boolean ultrafilters 
in a relation algebra and let $;$
denote composition. 
Then $$F;G=\{X;Y: X\in F, Y\in G\}.$$
The triple $(F,G,K)$ is {\it consistent} 
if the following holds:
$$F;G\subseteq K, F;K\subseteq G\text{  and }G;K\subseteq F.$$

So to represent $\cal R$ using $Uf$ as colours, we want to achieve (i) -(iii) below:

\begin{enumroman}

\item $(U^a,U^b,U^W)$ is consistent whenever $a,b\in H$ and $a;b\in U^W.$

\item $(F,G,K)$ is consistent whenever at least two of $F,G,K$ are non-principal
and $F,G,K\in Uf-\{U^{Id}\}.$

\item For any $a,b,c,d\in H$, there is $W\in J'$ such that $a;b\cap c;d\in U^W$.

\end{enumroman}

Let us see how to represent this algebra.
We call $(G,l)$ a  {\it consistent coloured graph} if $G$ is a set,
$l:G\times G\to Uf$
such that for all $x,y,z\in G$, the following hold:

\begin{enumroman}

\item $l(x,y)=U^{Id}$ iff $x=y,$

\item $l(x,y)=l(y,x)$

\item The triple $(l(x,y),l(x,z),l(y,z))$ is consistent.

\end{enumroman}

We say that a consistent coloured graph $(G,l)$ is complete if for all
$x,y\in G$, and $F,K\in Uf$, whenever
$(l(x,y),F,K)$ is consistent, then there is a node $z$ such that
$l(z,x)=F$ and $l(z,y)=K$.
We will build a complete consistent graph step-by-step. So assume (inductively)
that$(G,l)$ is a consistent coloured graph and $(l(x,y), F, K)$ is a consistent triple.
We shall extend $(G,l)$ with a new point $z$ such that $(l(x,y), l(z,x), l(z,y)) =(l(x,y),G,K).$
Let $z\notin G.$ We define $l(z,p)$ for $p\in G$ as follows:
$$l(z,x)=F$$
$$l(z,y)=K, \text { and if } p\in G\smallsetminus \{x,y\}, \text { then }$$
$$l(z,p)=U^W\text { for some }W\in J' \text { such that both }$$
$$(U^W, F, l(x,p))\text { and } (U^W, K, l(y,p))\text { are consistent }.$$
Such a $W$ exists by our assumptions (i)-(iii).
Conditions (i)-(ii) guarantee that this extension is again a consistent coloured graph.

We now show that any non-empty complete coloured graph $(G,l)$ gives a representation
for $\cal R.$ For any $X\in R$ define
$$rep(X)=\{(u,v)\in G\times G: X\in l(u,v)\}$$
We show that
$$rep:{\cal R}\to R(G)$$
is an embedding. $rep$ is a boolean homomorphism
because all the labels are ultrafilters.
$$rep(Id)=\{(u,u): u\in G\},$$
and for all $X\in R$,
$$rep(X)^{-1}=rep(X).$$
The latter follows from the first condition in the definition of a consistent coloured graph.
From the second condition in the definition of a consistent coloured graph, we have:
$$rep(X);rep(Y)\subseteq rep(X;Y).$$
Indeed, let $(u,v)\in rep(X), (v,w)\in rep(Y)$ I.e. $X\in l(u,v), Y\in l(v,w).$
Since $(l(u,v), l(v,w), l(u,w))$ is consistent, then $X;Y\in l(u,w)$, i.e. $(u,w)\in
rep(X;Y).$
On the other hand, since $(G,l)$ is complete and because (i)-(ii) hold, we have:
$$rep(X;Y)\subseteq rep(X);rep(Y),$$
because $(G,l)$ is complete and because (i) and (ii) hold.
Indeed, let $(u,v)\in rep(X;Y)$. Then $X;Y\in l(u,v)$.
We show that there are $F,K\in \Uf$ such that
$$X\in F, Y\in K \text { and } (l(u,v), F,K)\text { is consistent }.$$
We distinguish between two cases:

{\bf Case 1}. $l(u,v)=U^a$ for some $a\in F$. By $X;Y\in U^a$ we have $a\in X;Y.$
Then there are $b\in X$, $c\in Y$ with $a\leq b;c.$ Then $(U^a, U^b, U^c)$ is consistent.

{\bf Case 2.}
$l(u,v)=U^W$ for some $W\in J'$. Then $|X;Y\cap E^W|\geq \omega$
by $X;Y\in U^W$.
Now if both $X$ and $Y$ are finite, then there are $a\in X$, $b\in Y$ with
$|a;b\cap E^W|\geq \omega$.
Then $(U^W, U^a, U^b)$ is consistent by (i). Assume that one of $X,Y$, say $X$ is infinite.
Let $S\in J'$ such that $|X\cap E^S|\geq \omega$ and let $a\in Y$ be arbitrary.
Then $(U^W, U^S, U^a)$ is consistent by (ii)
and $X\in U^S, Y\in U^a.$

Finally, $rep$ is one to one because $rep(a)\neq \emptyset$ for all $a\in A$.
Indeed $(u,v)\in rep(Id)$ for any
$u\in G$. Let $a\in H$. Then $(U^{Id}, U^a, U^a)$ is consistent, so there is a $v\in G$ with
$l(u,v)=U^a$. Then $(u,v)\in rep(a).$

\end{enumarab}

\end{proof}

\subsection{The cylindric algebra}

We define the atom structure like we did before. The basic matrices of the atom structure above form a
$3$ dimensional cylindric algebra. We want an
$n$ dimensional one.  Our previous construction of the atom structure satisfied (*)
satisfies $(\forall a_1\ldots a_3\ b_1\ldots b_3\in I)(\exists W\in J)(a_1;b_1)\cap\ldots (a_3;b_3)\in U^W.$

We strengthen this condition to (**)
$$(\forall a_1\ldots a_nb_1\ldots b_n\in I)(\exists W\in J)W\cap (a_1;b_1)\cap\ldots  (a_n;b_n)\neq \emptyset.$$
(This is referred to in \cite{Sayed} as an {\it $n$ complex blur for $\bold M$},
our first construction was a $3$ complex blur).

This condition will entail that the set of {\it all} $n$ by $n$ matrices is a cylindric basis on the new relation
algebra ${\cal R}_n$ defined as before, with minor modifications.

Now ${\cal R}_n$ is defined by taking $I$ be a finite set with $|I|\geq 2n+2$, $J$ be the set of all $2$ (See the proof) element
subsets of $I$. And then define everything as before.
The resulting cylindric algebra is also called the blow up and blur {\it cylindric algebra of dimension $n$},
which actually blows up  and blurs the $n$ dimensional finite cylindric algebra consisting of $n$ basic matrices of $\bold M$, whch is representable,
so such  an algebra exists for every $n.$

The new condition (**)
guarantess the amalgamation property of matrices (corresponding to commutativity of
cylindrifiers)  which is the essential  property of basis.

We know that the term algebra is a subneat reduct of an algebra in $\omega$ extra dimensions.
But we need a final tick so that the the term cylindric algebra is a {\it full} neat reduct.
This requires a yet another strenghthenig of  (**) by replacing $\exists$ by $\forall$.

Now under this stronger condition, let $\B_n$ be the set of basic matrices of our blown up and blurred ${\cal R}_n$.
In the first order language $L$ of $(\omega, <)$, which has quantifier elimination,
diagrams are defined for each $K\subseteq n$ and $\phi\in L$, via maps  $\bold e:K\times K\to {\cal R}_n$.
For an atom let $v(a)$ be its ith co-ordinate, or its $i$ th level in the rectangle.

The pair $\bold e$ and $\phi$ defines an element in $\Cm\B_n$, called a diagram,
that is a set of matrices, defined by
$$M(\bold e, \phi)=\{m\in B_n, i,j\in K, m_{ij}\leq \phi(e_{ij},v(m_{ij})\}.$$

A normal diagram is one whose entries are either atoms or finitely many blurs (by (J)), that is elements of the form $E^W$,
in addition to the condition that $\phi$ implies $\phi_e$.
Any diagram can be approximated by normal ones; and atcually it is a finite union of normal diagrams.
The term algebra turns, denoted by $\Bb_n(\bold M, J, e)$, consists of those diagrams,
and finally we get  that  that for $t<n$
$$\Nr_t \Bb_n(\bold M, J, e)\cong \Bb_t(\bold M, J, e).$$
Here actually we are also blowing and bluring the finite dimensional cylindric algebra atom consisting of matrices on $\bold M$, we blow 
up every $n$ dimensional matrix to infinitely many, where each entry is either an atom of the relation algebra or a blur;
these are exactly 
the diagrams.

\subsection{The analogy, first informaly, then formally in a map}

This construction actually has a lot of affinity with the first  model theoretic construction.
First they both prove the same thing; the Andreka et all construction  proves that in addition the term algebra is a $k$ neat reduct.
Now here we are comparing a relation algebra construction with a cylindric algebra one, but the analogy is worthwhile pointing out.

Replace the clique $N$ in Hodkinson's construction by $\bold M$, in this case the term algebra, 
$\cal R$ is also obtained by replacing every atom by infinitely many ones, 
and $\bold M$ appears on the global level as a subalgebra of the complex algebra.

To this consruction we can also associate a finite graph with finite chromatic number, namely the complete graph on $\bold M$.
The blurs are the colours, that correspond to the colours $(\rho, i)$ in Sayed Ahmed 's construction.  

In the first case the splitting of the clique $N$, uses just one index, in the second we use two indices, the atoms of $\bold M$ and the blurs.
The first partition replaces the use of Ramseys theorem, the second partition, is a devision of the whole splitting into finitely many 
rectangles, one for each blur.  The homogeneous model $M$ in the second construction correspond to the second partition,  in the sense 
that it is {\it not}  the base of the representable term algebra,
but $W$ is, which is basically obtained by removing the blurs, that are the same time 
essential in representing it, $W$ thus corresponds to the term algebra of co-finite finite intersections with the second partition, 
which in turn is representable.

In short, we start up with a finite structure, blow it up, on the term algebra level, using blurs to represent it, 
but it will not be blurred enough to disappear on the complex algebra level, forcing
the latter to be non-representable (due to incompatibility of "a finitenes condition") with the inevitability 
of representing the complex algebra on an infinite set.

More formally, we define a function that maps the ingredients of the first construction to that of the second:

$$N\mapsto \bold M$$
$$N\times \omega\times n\mapsto \omega\times P\times J$$
In the former case $J'=\emptyset$, the blurs do not appear on this level, in the second spliting all blurs are used.
$$\{(\rho , i): i<n \}\mapsto \{ W: W\in J\}.$$
Here in the first case $n$ blurs are needed to represent the new term algebra. In the later it is the number of two elements subsets of 
$J$.
For $\phi\in L^+$, let $\phi^M=\{s\in {}^nM: M\models \phi[s]\}.$
Here we are {\it not} relativizing semantics.
$$\{{\phi}^M: \phi \in L\}\mapsto \Cof(E^W)$$

$$\A\mapsto {\cal R}$$

$$\Cm\At\A\to \Cm\At{\cal R}$$

Here we include more examples.

\begin{example}

Let $l\in \omega$, $l\geq 2$, and let $\mu$ be a non-zero cardinal. Let $I$ be a finite set,
$|I|\geq 3l.$ Let
$$J=\{(X,n): X\subseteq I, |X|=l,n<\mu\}.$$
Let $H$ be as before, i.e.
$$H=\{a_i^{P,W}: i\in \omega, P\in I, W\in J\}.$$
Define
$(a_i^{P,S,p}, a_j^{Q,Z,q}, a_k^{R,W,r})$ is consistent ff

$S\cap Z\cap W=\emptyset$ or
$ e(i,j,k) \text { and } |\{P,Q,R\}|\neq 1.$

Pending on $l$ and $\mu$, let us call these atom structures ${\cal F}(l,\mu).$
Then our first  example in is just
${\cal F}(2,1).$

If $\mu\geq \omega$, then $J$ as defined above would be infinite,
and $\Uf$ will be a proper subset of the ultrafilters.
It is not difficult to show that if $l\geq \omega$
(and we relax the condition that $I$ be finite), then
$\Cm{\cal F}(l,\mu)$ is completely representable,
and if $l<\omega$ then $\Cm{\cal F}(l,\mu)$ is not representable. In the former case we have infnitely many colours, so that the chromatic nmber
of the graph is infinite, while in the second case the chromatic number is infinite.

Informally, if the blurs get arbitarily large, then in the limit, the resulting algebra will be completely representable, and so its complex algebra
will be representable. If we take, a sequence of blurs, each finite, but increasing in size we get a sequence of algebras
that are not completely representable, and the sequence of their complex algebras will not be representable. The limit of the former,
will be  completely representable (with  an infinite set of blurs); its completion will be the limit of the second sequence of non representable, 
will be representable. Either construction can be used to achieve this. 

This phenomena has many reincarnations in the literature. One is the following:
It is is nothing more than Monk's classical non finite
axiomatizability result;  it gives a sequence of non representable algebras whose ultraproduct is completely representable.
\end{example}
Using such examples, we now prove:
\begin{corollary}
\begin{enumarab}
\item  The classes $\RRA$ is not finitely axiomatizable.
\item  The elementary
closure of the class ${\bf CRA}$ is not finitely axiomatizable.
\end{enumarab}
\end{corollary}
\begin{demo}{Proof}
For the second  we use the second construction.  Let ${\cal D}$ be a non-
trivial ultraproduct of the atom structures ${\cal F}(i,1)$, $i\in \omega$. Then $\Cm{\cal D}$
is completely representable. 
Thus $\Tm{\cal F}(i,1)$ are $\RRA$'s 
without a complete representation while their ultraproduct has a complete representation.
Also $\Cm{\cal F}(i,1)$, $i\in \omega$ 
are non representable with a completely representable ultraproduct.
This yields the desired result.

We prove the cylindric case. Take $\G_i$ to be the disjoint union of cliques of size $n(n-1)/2+i$. 
Let $\alpha_i$ be the corresponding atom astructure
of $\A_i$, as constructed above. Then $\Cm \A_i$ is not representable, but $\prod_{i\in \omega}\Cm\A_i=\Cm(\prod_{i\in \omega}\A_i)$. 
Then the latter is based on the disjoint union of the cliques which is arbitrarily large, hence is representable. 
\end{demo}
The first construction also works, by using relation algebra atom structures with $n$ dimensional cylindric bases, this will
yield the analogous result for cylindric algebras.

The second re-incarnation is due to Hirsch and Hodkinson, it also works for relation and cylindric algebras, and this is the essence. For each graph
$\Gamma$, they associate a cylindric algebra atom structure of dimension $n$, $\M(\Gamma)$ such that $\Cm\M(\Gamma)$ is representable
if and only if the chomatic number of $\Gamma$, in symbols $\chi(\Gamma)$, which is the least number of colours needed, $\chi(\Gamma)$ is infinite. 
Using a famous theorem of Erdos, they construct  a sequence $\Gamma_r$ with infinite chromatic number and finite girth, 
whose limit is just $2$ colourable, they show that the class of strongly representable
algebras  is not elementary. Notice that this is a {\it reverse process} of Monk-like  constructions, given above, 
which gives a sequence of graphs of finite chromatic number whose limit (ultarproduct) has infinite
chromatic number.

And indeed, the construction also, 
is a reverse to Monk's construction in the following sense:
Some statement fail in $\A$ iff $At\A$ 
be partitioned into finitely many $\A$-definable sets with certain 
`bad' properties. Call this a {\it bad partition}. 
A bad partition of a graph is a finite colouring. So Monks result finds a sequence of badly partitioned atom structures,
converging to one that is not. As we did above, this boils down, to finding graphs of finite chromatic numbers $\Gamma_i$, having an ultraproduct
$\Gamma$ with infinite chromatic number. 

An atom structure is {\it strongly representable} iff it 
has {\it no bad partition using any sets at all}. So, here, the idea  find atom structures, with no bad partitions, 
with an ultraproduct that does have a bad partition.
From a graph Hirsch nad Hodkinson constructed  an atom structure that is strongly representable iff the graph
has no finite colouring.  So the problem that remains is to find a sequence of graphs with no finite colouring, 
with an ultraproduct that does have a finite colouring, that is, graphs of infinite chromatic numbers, having an ultraproduct
with finite chromatic number.

 It is not obvious, a priori, that such graphs actually exist.
And here is where Erdos' methods offer solace. Indeed, graphs like this can be found using the probabilistic methods of Erdos, for those methods
render finite graphs of arbitrarily large chormatic number and girth. 
By taking disjoint unions, one {\it can get}  graphs of infinite chromatic number (no bad partitions) and arbitarly large girth. A non principal 
ultraproduct of these has no cycles, so has chromatic number 2 (bad partition).



\begin{thebibliography}{}




\bibitem{1} Andr\'eka, Ferenczi, N\'emeti (Editors) {\bf Cylindric-like Algebras and Algebraic Logic},
Andr\'eka, Ferenczi, N\'emeti (Editors) Bolyai Society Mathematical Studies p.205-222 (2013).

\bibitem{sayed}H. Andr\'eka, I.  N\'emeti, T. Sayed Ahmed,
{\it Omitting types for finite variable fragments and complete representations of algebras,}
Journal of Symbolic Logic, {\bf 73}(1) (2008), p.65-89.atical Society, Volume 3, Number 4, October 1990, Pages 903-928.
\bibitem{HMT1} L. Henkin, J.D. Monk and  A.Tarski, {\it Cylindric Algebras Part I}.
North Holland, 1971.
\bibitem{HMT2} L. Henkin, J.D. Monk and  A.Tarski, {\it Cylindric Algebras Part II}.
North Holland, 1985.
\bibitem{r} R. Hirsch, {\it Relation algebra reducts of cylindric algebras and complete representations}, The Journal of Symbolic Logic, Vol. 72, Number 2, June 2007.
\bibitem {AU} I. Hodkinson, \emph{A construction of cylindric and polyadic algebras from atomic relation algebras}, Algebra Universalis, \textbf{68} (2012), pp. 257-285.
\bibitem{OTT} M. Khaled and T. Sayed Ahmed, {\it Omitting types algebraically via cylindric algebras}, International Journal of Algebra, Vol. 3 (2009), no. 8, pp. 377 - 390.
\bibitem{Khaled} M. Khalid and T. Sayed Ahmed,  \emph{Vaughts theorem holds for} $L_2$ \emph{but fails for} $L_n$ \emph{when} $n>2$, Bulletin of the Section of Logic, Volume 39:3/4 (2010), pp. 107-122.
\bibitem{69} L. Mayer, \emph{Vaught's Conjecture for o-minimal theories}, Journal of Symbolic Logic 53 (1988), 146-159.
\bibitem{Nemeti} I. N\'emeti, {\it On cylindric algebraic model theory}: In {\bf Algebraic logic and Universal Algebra in Computer Science} (Proc. Conf. Ames 1988). Editiors: C. H. Bergman, R. D. Maddux and D. L. Pigozzi, Springer-Verlag, Berlin (1990), pp. 37-76.
\bibitem{Sagi} G. S\'agi and D. Szir\'aki, \emph{Some Variants of Vaught's Conjecture from the Perspective of Algebraic Logic}, Logic Journal of the IGPL, published online January 5, 2012.
\bibitem {Samir} T. Sayed Ahmed and B. Samir, {\it  Omitting types for first order logic with infinitary predicates}, Math Logic Quarterly (2007), p. 564-576.
\bibitem{Studia} T. Sayed Ahmed, {\it Martin's axiom, Omiting types and Complete represenations in algebraic logic} Studai Logica, \textbf{72} (2002), pp. 285-309.
\bibitem{FM} T. Sayed Ahmed, {\it  The class of neat reducts of polyadic algebras  is not elementary}, Fundementa Mathematica, \textbf{172}(2002), pp. 61-81
\bibitem{MLQ} T. Sayed Ahmed, {\it A model-theoretic solution to a problem of Tarski}, Math Logic Quarterly, \textbf{48} (2002), pp. 343-355.
\bibitem{BSL} T. Sayed Ahmed, \emph{Neat embedding is not sufficient for complete representability}, Bulletin of the Section of Logic, Volume 36:1/2 (2007), pp. 21-27.
\bibitem {Sayed2} T. Sayed Ahmed, {\it An interpolation Theorem for first order logic with infinitary predicates}, Logic Journal of the IGPL (2007) p.  21-32.
\bibitem{Sayed} T. Sayed Ahmed, {\it Completions, Complete representations and Omitting types} in \cite{1}
In: {\bf Cylindric-like Algebras and Algebraic Logic}. Editors: H. Andr\'eka, M. Ferenczi, I. N\'emeti, Bolyai Society Mathematical Studies, pp.205-222 (2013).

\bibitem{Simon} A. Simon, {\it A completeness theorem for typless logics}, In Algebraic Logic, Andreak, Monk, Nemeti (editors)
\bibitem {HHbook} Hirch Hodkinson {\it Completion and complete representations in algebraic logic } in \cite{1}
\bibitem{Vaught} Sayed Ahmed {\it On a theorem of Vaught for finite variable fragments} Journal of  applied  classical logics (2009) 97-112
\bibitem{NeatIGPL} Sayed Ahmed Basim Samir {\it A neat embeding theorem for expansions of cylindric algebras} IGPL
\bibitem{Sayedneat} Sayed Ahmed {\it Neat reducts and neat embedings in cylindric algebras} in \cite{1}
\bibitem{Biro} Biro {\it Non finite axiomatizability results in algebrac logc } Journal of Symbolic Logic (1992) 832-843
\bibitem{t} Sayed Ahmed { The class $SNr$ is not closed under completions} Logic Journal of IGPL
\bibitem{HHbook2} Hirsch, Hodkinson {\it Relation algebras by games}
\end{thebibliography}
\end{document}